\documentclass{amsart}
\usepackage{amssymb}
\usepackage{stmaryrd}
\usepackage{fullpage}
\usepackage[colorlinks]{hyperref}

\newtheorem{theorem}{Theorem}[section]
\newtheorem{lemma}[theorem]{Lemma}
\newtheorem{cor}[theorem]{Corollary}
\newtheorem{prop}[theorem]{Proposition}

\theoremstyle{definition}
\newtheorem{definition}[theorem]{Definition}

\theoremstyle{remark}
\newtheorem{remark}[theorem]{\bf Remark}

\numberwithin{equation}{section}

\newcommand{\FF}{{\mathbb{F}}}

\newcommand{\bC}{{\mathbf{C}}}
\newcommand{\bZ}{{\mathbf{Z}}}
\newcommand{\bF}{{\mathbf{F}}}
\newcommand{\bN}{{\mathbf{N}}}

\newcommand{\Aut}{{\operatorname{Aut}}}

\newcommand{\Irr}{{\operatorname{Irr}}}

\newcommand{\Syl}{{\operatorname{Syl}}}

\newcommand{\Sp}{{\operatorname{Sp}}}
\newcommand{\SL}{{\operatorname{SL}}}

\newcommand{\GL}{{\operatorname{GL}}}
\newcommand{\Gal}{{\operatorname{Gal}}}

\newcommand{\stab}{{\operatorname{stab}}}
\newcommand{\Out}{{\operatorname{Out}}}
\newcommand{\Ortho}{{\operatorname{O}}}
\newcommand{\GF}{\mbox{GF}}

\newcommand{\St}{{\rm St}}

\let\nor=\triangleleft
\begin{document}

\title{On Gluck's conjecture}

\author{Yong Yang}
\address{Department of Mathematics, Texas State University, 601 University Drive, San Marcos, TX 78666, USA}


\makeatletter
\email{yang@txstate.edu}
\makeatother
\Large

\date{}



\begin{abstract}
This paper presents the best known bounds for a conjecture of Gluck and a conjecture of Navarro.


\end{abstract}

\maketitle
\section{Introduction} \label{sec:introduction8}

Let $G$ be a finite group and we denote by $b(G)$ the largest degree of an irreducible character of $G$. Motivated by a result of Jordan, Isaacs and Passman (see ~\cite[Theorem 12.26]{IMIB}) showed that if $G$ is nilpotent, then $G$ has an abelian subgroup of index at most $b(G)^4$. In ~\cite{GLUCK} Gluck proved that in all finite groups the index of the Fitting subgroup $\bF(G)$ in $G$ is bounded by a polynomial function of $b(G)$. For solvable groups, Gluck further showed that $|G : \bF(G)| \leq b(G)^{13/2}$ and conjectured that $|G : \bF(G)| \leq b(G)^2$. Gluck's conjecture has been studied extensively for the past 30 years and is still open as far as we know. 

The conjecture has been verified by Espuelas ~\cite{AE1} for groups of odd order. Espuelas' result has been extended in ~\cite{DolfiJ} to solvable groups with abelian Sylow $2$-subgroups by Dolfi and Jabara. The author has verified the conjecture for all solvable groups with order not divisible by $3$ ~\cite{YY4}. We should mention that recently Cossey, Halasi, Mar\'oti, and Nguyen ~\cite{Gluckgeneral} provided a modest generalizations of the previously mentioned results, and also obtained a nice bound for arbitrary finite groups. As of today, for solvable groups, the best general bound $|G : \bF(G)| \leq b(G)^{3}$ has been given by Moret\'o and Wolf ~\cite{MOWOLF}.

In this paper, we will focus on the original Gluck's conjecture, and the main goal of this paper is to obtain the best known bound of Gluck's conjecture for all solvable groups. In fact we show that $|G: \bF(G)| \leq b^{\alpha}(G)$ where $\alpha=\frac {\ln (6 \cdot (24)^{1/3})}{\ln 3} \approx 2.595$. We also improve an upper bound with respect to a conjecture of Navarro about the product of the largest character degrees of Sylow subgroups. In order to achieve these goals, we improve the bound for the order of solvable linear groups with respect to the largest orbit size of the group action.

Many questions in the theory of finite groups eventually come down to questions about the orbit structure of the group action on finite modules. One of the most important and natural questions about orbit structure is to establish the existence of an orbit of a certain size. Obviously, the size of an orbit can never exceed the order of the group. Let $V$ be a $G$-module and let $v\in V$ be such that $\bC_G(v)=1$. In this situation we say that $v^G$ is a regular orbit of $G$ on $V$. For the existence of regular orbits of primitive solvable linear groups, please see ~\cite{YY2,YY3}. For a solvable linear group $G$, we denote the largest orbit size of $G$ on the vector space $V$ by $M$. Note that the bounds one obtained before are always an integer power of $M$ ~\cite{Gluckgeneral,DolfiJ,AE1,MOWOLF}, and this is due to the fact that one always tries to prove the existence of regular orbits of $G$ on $V$, or $V \oplus V$, or $V \oplus V \oplus V$. For example, suppose one can show that $G$ has a regular orbit on $V \oplus V \oplus V$, then this would imply that there exists an orbit of size greater than or equal to $\sqrt[3]{|G|}$. We believe this approach perhaps was influenced by the study of the bases of permutation group actions. However, we remark here that the existence of regular orbit is stronger than the arithmetical condition on orbit size but relatively convenient to show using induction.

The approach in this paper is build on some preliminary work the author has done in ~\cite{YYodd}. One key point is to study carefully how a permutation group acts on the power set, and to see if certain arithmetical properties of the orbit structure hold. Fortunately for solvable groups we have found a way to do so. We also remark here that based on an example in ~\cite{WOLF1} (in fact, one can construct many others), it is not true that for all solvable linear groups, one has that $|G| \leq M^2$.







\section{Notation and Lemmas} \label{sec:Notation and Lemmas}
We fix some notation and prove some preliminary results in this section.

Notation:
\begin{enumerate}




\item We set $\lambda=(24)^{1/3}$, $\alpha=\frac {\ln (6 \cdot (24)^{1/3})}{\ln 3} \approx 2.595$, and $\beta=\frac {2.25} {3.25} \cdot {\alpha} \approx 1.7967$.

\item If $V$ is a finite vector space of dimension $n$ over $\GF(q)$, where $q$ is a prime power, we denote by $\Gamma(q^n)=\Gamma(V)$ the semilinear group of $V$, i.e.,
\[\Gamma(q^n)=\{x \mapsto ax^{\sigma}\ |\ x \in \GF(q^n), a \in \GF(q^n)^{\times}, \sigma \in \Gal(\GF(q^n)/\GF(q))\},\] and we define \[\Gamma_0(q^n)=\{x \mapsto ax\ | \ x \in \GF(q^n), a \in \GF(q^n)^{\times}\}.\]

\end{enumerate}

We recall that an irreducible $G$-module $V$ is called quasi-primitive if $V_N$ is homogeneous for all $N \triangleleft G$.

\begin{definition} \label{defineEi}
Suppose that a finite solvable group $G$ acts faithfully, irreducibly and quasi-primitively on a finite vector space $V$. Let $\bF(G)$ be the Fitting subgroup of $G$ and $\bF(G)=\prod_i P_i$, $i=1, \dots, m$ where $P_i$ are normal $p_i$-subgroups of $G$ for different primes $p_i$. Let $Z_i = \Omega_1(\bZ(P_i))$. We define \[E_i=\left\{ \begin{array}{lll} \Omega_1(P_i) & \mbox{if $p_i$ is odd}; \\ \lbrack P_i,G,\cdots, G \rbrack & \mbox{if $p_i=2$ and $\lbrack P_i,G,\cdots, G \rbrack \neq 1$}; \\  Z_i & \mbox{otherwise}. \end{array} \right.\] By proper reordering we may assume that $E_i \neq Z_i$ for $i=1, \dots, s$, $0 \leq s \leq m$ and $E_i=Z_i$ for $i=s+1, \dots, m$. Define $E=\prod_{i=1}^s E_i$, $Z=\prod_{i=1}^s Z_i$ and we define $\bar{E}_i=E_i/Z_i$, $\bar{E}=E/Z$. Furthermore, we define $e_i=\sqrt {|E_i/Z_i|}$ for $i=1, \dots, s$ and $e=\sqrt{|E/Z|}$.
\end{definition}

\begin{theorem} \label{Strofprimitive}
Suppose that a finite solvable group $G$ acts faithfully, irreducibly and quasi-primitively on an $n$-dimensional finite vector space $V$ over finite field $\FF$ of characteristic $r$. We use the notation in Definition ~\ref{defineEi}. Then every normal abelian subgroup of $G$ is cyclic and $G$ has normal subgroups $Z \leq U \leq F \leq A \leq G$ such that,
\begin{enumerate}
\item $F=EU$ is a central product where $Z=E \cap U=\bZ(E)$ and $\bC_G(F) \leq F$;
\item $F/U \cong E/Z$ is a direct sum of completely reducible $G/F$-modules;
\item $E_i$ is an extra-special $p_i$-group for $i=1,\dots,s$ and $e_i=p_i^{n_i}$ for some $n_i \geq 1$. Furthermore $(e_i,e_j)=1$ when $i \neq j$ and $e=e_1 \dots e_s$ divides $n$, also $\gcd(r,e)=1$;
\item $A=\bC_G(U)$ and $G/A \lesssim \Aut(U)$, $A/F$ acts faithfully on $E/Z$;
\item $A/\bC_A(E_i/Z_i) \lesssim \Sp(2n_i,p_i)$;
\item $U$ is cyclic and acts fixed point freely on $W$ where $W$ is an irreducible submodule of $V_U$;
\item $|V|=|W|^{eb}$ for some integer $b$ and $|G:A| \mid \dim(W)$.
\item $G/A$ is cyclic.
\end{enumerate}
\end{theorem}
\begin{proof}
This is ~\cite[Theorem 2.2]{YY2}.
\end{proof}

The case where $G$ is acting primitively on $V$ is a central case that one returns to again and again in the theory of solvable groups. In short, as indicated in Definition ~\ref{defineEi}, such groups have an invariant $e$ that measures their complexity. It is known that if $e > 118$, then $G$ will have a regular orbit ~\cite{manz/wolf}. Some early results of the author were able to improve this result by classifying all the cases with respect to $e$ when the regular orbit will exist. The following classifications appeared in \cite{YY2} and \cite{YY3}.

\begin{theorem} \label{quote1}
Suppose that a finite solvable group $G$ acts faithfully, irreducibly and quasi-primitively on a finite vector space $V$. By Definition ~\ref{defineEi} and Theorem ~\ref{Strofprimitive}, $G$ will have a uniquely determined normal subgroup $E$ which is a direct product of extra-special $p$-groups for various $p$ and $e=\sqrt{|E/\bZ(E)|}$. Assume $e=5,6,7$ or $e \geq 10$ and $e \neq 16$, then $G$ will have at least two regular orbits on $V$.
\end{theorem}
\begin{proof}
This follows from \cite[Theorem 3.1]{YY2} and \cite[Theorem 3.1]{YY3}.
\end{proof}


We now state some known results about solvable permutation groups.


\begin{lemma}  \label{primitiveperm}
If $G$ is a solvable primitive permutation group on $\Omega$, then $G$ has a unique minimal normal subgroup $V$ with $|V|=|\Omega|=p^t$ for a prime $p$, $G=V S$ where $S$ is a point stabilizer, and $V$ is a faithful irreducible $S$-module.
\end{lemma}
\begin{proof}
This is a well-known result.
\end{proof}

\begin{lemma}
Let $G$ be a finite solvable permutation group on a set $\Omega$ where $|\Omega|=n=p^t$. Assume $g \in G^{\#}$, and let $o(g)$ be the smallest prime divisor of the order of $g$. We denote by $n(g)$ the numbers of cycles of the action of $g$ on $\Omega$, and by $s(g)$ the number of fixed points. Then $n(g) \leq (n+s(g))/2 \leq (p+o(g)-1)n/(o(g) p) \leq 3n/4$.
\end{lemma}
\begin{proof}

Let $V$ be a minimal normal subgroup of $G$ and $S$ denote a point stabilizer. Then by Lemma ~\ref{primitiveperm}, we know that $G=V S$, and $|V|=n=p^t$. If $s(g)=0$, we clearly have $n(g) \leq n/2$. We thus may assume that $g$ has fixed points, and without loss of generality we may assume $g \in S$. Since the actions of $S$ on $V$ and $\Omega$ are permutation isomorphic, it follows that $s(g)=|\bC_V(g)|$ and since $S$ acts faithfully on $V$, $s(g) \mid |V|/p =n/p$.

Therefore,
\[n(g) \leq s(g)+(n-s(g))/o(g)=(n+(o(g)-1)s(g))/o(g) \leq (p+o(g)-1)n/(o(g) p) \leq 3n/4.\]
\end{proof}


\begin{lemma}  \label{lem4}
Let $G$ be a solvable permutation group of degree $n$. Then $|G| \leq \lambda^{n-1}$.
\end{lemma}
\begin{proof}
This result is well-known (see ~\cite[Theorem 5.8B]{Dixon-Mortimer}).
\end{proof}

Since in our situation we need to compare the orbit size and the group order rather than proving the existence of regular orbits, we need some quantitative results about how the solvable permutation group acts on the power set of the base set. The following can be viewed as the quantitative analogue of Gluck's lemma  ~\cite[Theorem 5.6]{manz/wolf} about primitive permutation groups.

\begin{theorem} \label{thmperm2}
Let $G$ be a solvable primitive permutation group of degree $n$ that acts on a set $\Omega$ and $|G|>1$. If $k$ is the smallest integer such that $|G| \le \lambda^k$ then for any subset $\Lambda$ of size $0 \leq m \leq k$, there exists $\Delta \subseteq \Omega-\Lambda$ such that $|G:\stab_G(\Delta)|^{\alpha} \cdot \lambda^{m-1} \ge |G|$.
\end{theorem}
\begin{proof}
Let $g$ be an element in $G$. We denote by $n(g)$ the number of cycles of $g$ on $\Omega$, by $o(g)$ the smallest prime divisor of the order of $g$, and by $s(g)$ the number of fixed points of $g$. Since $G$ is solvable, $n$ is a power of a prime, and we set $n=p^t$ where $p$ is a prime. We set $o$ to be the smallest element order of all the nontrivial elements in $G$.

For a subset $X \subseteq G$, consider the following set.

$S(X)=\{(g,\Gamma) \ |\ g \in X, \Gamma \subseteq \Omega, g \in \stab_G(\Gamma)\}.$

It is easy to see that in order to estimate $S(G^{\#})$, we only need to count elements of prime order. Also, assume $g_1$ is of prime order and let $g_2 \in \langle g_1 \rangle$, then the subsets of $\Omega$ stabilized by $g_1$ and by $g_2$ are exactly the same. Note that $g \in G$ stabilizes exactly $2^{n(g)}$ subsets of $\Omega$.

Note that $g \in G$ stabilize exactly $2^{n(g)}$ subsets of $\Omega$. Consequently,

\[|S(G^{\#})| \leq |G| \cdot 2^{\lfloor (p+o-1)n/(o \cdot p) \rfloor}. \]

We know that $|G| \leq (1/\lambda) n^{\frac {13} 4}$ by ~\cite[Theorem 3.5(c)]{manz/wolf}.

Thus, we have \[|S(G^{\#})| \leq (1/\lambda) n^{\frac {13} 4} \cdot 2^{\lfloor \frac {3 n} 4 \rfloor}.\]

Let $|G|=(\lambda)^l$, thus $k=\lceil l \rceil$ (the ceiling function).


In almost all cases, we can show that for any subset $\Omega_1 \subseteq \Omega$ with $n-k$ elements, there exists a subset $\Delta \subseteq \Omega_1$ such that $G$ has a regular orbit on $\Delta$ (as $G$ acts on the power set of $\Omega$), and thus the result holds. The spirit of the proof is, look at any subsets of $\Omega$ of size $n-k$, call it $\Omega_1$, by throwing away enough bad subsets out of the power set of $\Omega_1$, there is still enough room left so there exists a subset $\Delta \subseteq \Omega_1$ such that $G$ has a regular orbit on it. In the exceptional cases, we have to do case analysis.


In order to show that for any subset, say $\Omega_1 \subseteq \Omega$ with $n-k$ elements, there exists a subset $\Delta \subseteq \Omega_1$ such that $G$ has a regular orbit on $\Delta$, it suffices to show that $|S(G^{\#})| <2^{n-k}.$




In view of the previous estimations, it suffices to show that \[|G| \cdot 2^{\lfloor (p+o-1)n/(o \cdot p) \rfloor} < 2^{n-k}.\]

or

\[\frac 1 {\lambda} \cdot n^{\frac {13} 4} \cdot 2^{\lfloor \frac {3 n} 4 \rfloor} < 2^{n-k}.\]

Assume $t>1$, we need to consider $n=p^t \leq 125$.

Assume $t=1$, we need to consider $n=p^t \leq 43$.

We remark that when $m=0$, we only need to consider $n=p^t \leq 9$ by ~\cite[Theorem 5.6]{manz/wolf}.

We use GAP ~\cite{GAP} to check the remaining cases.
\end{proof}

\section{Main Theorems} \label{sec:maintheorem}

Before we prove our main result, we review a standard technique and a little bit of the history of research in this direction.

\begin{lemma} \label{lemeasy}
Suppose that $G$ is a finite group and $V$ is a faithful $G$-module. Assume $G$ has a regular orbit on $V \oplus V$, then there exists $v \in V$ such that $|\bC_G(v)| \leq \sqrt{|G|}$.
\end{lemma}
\begin{proof}
There is an element $(v,u) \in V \oplus V$ such that $\bC_G((v, u))= \bC_G(v) \cap \bC_G(u)=1$.

Note that \[|\bC_G(v)| \cdot |\bC_G(u)| = \frac {|\bC_G(v)| \cdot |\bC_G(u)|} {|\bC_G(v) \cap \bC_G(u)|}=|\bC_G(v)  \bC_G(u)| \leq |G|.\]

It follows that, either  $|\bC_G(v)| \leq \sqrt{|G|}$ or $|\bC_G(u)| \leq \sqrt{|G|}$. 
\end{proof}

Using a result of Seress ~\cite{Seress}, Moret\'o and Wolf proved the following result. Let $V$ be a faithful completely reducible $G$-module (possibly of mixed characteristic) and assume that $G$ is solvable. Then there exist $u, v, w \in V$ such that $\bC_G(u) \cap \bC_G(v) \cap \bC_G(w)=1$. In particular, there exists $x \in V$ such that $|\bC_G(x)| \leq |G|^{1/3}$. This approach probably is influenced by the study of the base size of linear group actions. In order to prove the existence of a large orbit, one would try to find the existence of regular orbits on two or more copies of $V$. The strategy is proven to be useful in many cases. However, as indicated in the proof of Lemma ~\ref{lemeasy}, the statement about large orbit size is derived from a statement about base size but on the other hand the converse of Lemma ~\ref{lemeasy} is not true, and thus possibly the large orbit result could be further improved. In this paper, we use more detailed calculations to improve some of the known bounds.


\begin{prop} \label{prop1}
Let $G$ be a solvable primitive subgroup of $\GL(n,r)$, $r$ a prime number, $n$ a positive integer, and let $V$ be the natural module for $G$. Then $G$ has at least $r$ regular orbits on $V \oplus V$, unless $G$ is one of the following groups:
\begin{enumerate}
\item $\GL(2,2)$;
\item $\SL(2,3)$ or $\GL(2,3)$;
\item $3^{1+2}.\SL(2,3)$ or $3^{1+2}.\GL(2,3) \leq \GL(6,2)$;
\item $(Q_8 \Ydown Q_8).K \leq \GL(4,3)$ where $K$ is isomorphic to a subgroup of index $1$, $2$ or $4$ of $\Ortho^{+}(4,2)$.
\end{enumerate}
\end{prop}
\begin{proof}
This is \cite[Theorem 3.4]{DOLFI}.
\end{proof}

This implies that except for those exceptional cases, we have $|G| \leq M^2$, and thus as long as $M \geq 7$, we have that $\lambda \cdot |G| \leq M^{\alpha}$.

\begin{prop} \label{proporbit}
Let $G$ be a solvable primitive subgroup of $\GL(n,r)$. Assume $G \neq 1$, then $\lambda \cdot |G| \leq M^{\alpha}$
\end{prop}
\begin{proof}
We first assume $G$ is not one of the exceptional cases in Proposition ~\ref{prop1}.

We follow the notation in Theorem ~\ref{Strofprimitive}. By Theorem ~\ref{quote1}, we may assume that $e=1,2,3,4,8,9,16$. By the remark after Proposition ~\ref{prop1}, we just need to find an orbit of size greater than $7$.

Assume $e=16$, $E$ has a regular orbit on $V$ and thus $G$ has an orbit of size at least $2^9$, and the result follows. 

Assume $e=9$, $E$ has a regular orbit on $V$ and thus $G$ has an orbit of size at least $3^5$, and the result follows. 

Assume $e=8$, $E$ has a regular orbit on $V$ and thus $G$ has an orbit of size at least $2^7$, and the result follows. 

Assume $e=4$, $E$ has a regular orbit on $V$ and thus $G$ has an orbit of size at least $2^5$, and the result follows. 

Assume $e=3$, $E$ has a regular orbit on $V$ and thus $G$ has an orbit of size at least $3^3$, and the result follows. 

Assume $e=2$, $E$ has a regular orbit on $V$ and thus $G$ has an orbit of size at least $8$, and the result follows. 

Assume $e=1$, we have $G \leq \Gamma(V) = \Gamma(p^n)$ by ~\cite[Corollary 2.3(b)]{manz/wolf}. Then we may set $|V|=p^n$ and $|G|=|G/C||C|$ where $|G/C| \mid n$ and $|C| \mid |V|-1$. For $g \in G \backslash C$, $|\bC_V(g)| \leq |V|^{1/2}$, and for $g \in C$, $|\bC_V(g)| \leq |V|^{1/2}$. If $|G| < |V|^{1/2}$, then $G$ has a regular orbit on $V$ and the result is clear. Thus, we may assume that $|G|\geq |V|^{1/2}$. Since $C$ has a regular orbit on $V$, it suffices to show that $|C|^{\alpha} \geq \lambda |G|$, i.e. $|G|^{\alpha-1} \geq \lambda |G/C|^{\alpha}$.

Thus, it suffices to show that $(p^n)^{\frac 1 2 \cdot (\alpha-1)} \geq \lambda (n)^{\alpha}$. This is true unless $|V|=p^n$ satisfies the following conditions.

\begin{center}
\begin{tabular}{|c | c| }
\hline
 $p$ & $n$ \\
\hline
 $p=2$ & $n \leq 15$ \\
 $p=3$ & $n \leq 6$ \\
 $p=4$ & $n \leq 4$ \\
 $p=5$ & $n \leq 3$ \\
 \hline
\end{tabular}
\end{center}

We now examine those exceptional cases of $p^n$, and use a slightly more delicate counting method.


For $g \in G \backslash C$, $|\bC_V(g)| \leq |V|^{1/o(g)}$. Let $k$ be the different prime divisors of $|G/C|$ and $o$ be the smallest prime divisors of $|G/C|$ if $|C| \leq \frac 1 k |V|^{1/o}$, then $G$ has a regular orbit on $V$ and the result is clear. Thus, we may assume that $|C|\geq \frac 1 k |V|^{1/o}$. Since $C$ has a regular orbit on $V$, it suffices to show that $|C|^{\alpha} \geq \lambda |G|$, i.e. $|C|^{\alpha-1} \geq \lambda |G/C|$. Thus, it suffices to show that $(\frac 1 k p^{\frac n o})^{(\alpha-1)} \geq \lambda n$.

Using this counting argument, and by some calculations, one can check that those previously mentioned exceptional cases all satisfy the conclusion of the statement. We notice that for smaller cases, one needs to check the group structure directly but all shall work out similarly, either $C$ is small enough so that $G$ has a regular orbit on $V$, or $C$ is big enough so $G$ has a large enough orbit size.

We now assume $G$ is one of the exceptional cases in Proposition ~\ref{prop1}.
\begin{enumerate}
\item $\GL(2,2)$. We know that $M=3$ and $|G|=6$, and the result follows by a direct calculation.
\item $\SL(2,3)$ or $\GL(2,3)$. We know that $M=8$ and $|G|=24$ or $48$, and the result follows by a direct calculation.
\item $3^{1+2}.\SL(2,3)$ or $3^{1+2}.\GL(2,3) \leq \GL(6,2)$; We know that $M \geq 27$ and $|G| \leq 27 \cdot 48$, and the result follows by a direct calculation.
\item $(Q_8 \Ydown Q_8).K \leq \GL(4,3)$ where $K$ is isomorphic to a subgroup of index $1$, $2$ or $4$ of $\Ortho^{+}(4,2)$. We know that $M \geq 32$ and $|G| \leq 32 \cdot 72$, and the result follows by a direct calculation.
\end{enumerate}
\end{proof}



The author expects there should be a way to directly compare the largest orbit size with the group order, and this is the main motivation of this work. We have the following result.

\begin{theorem} \label{thm1}
  Suppose that $G$ is a finite solvable group and $V$ is a finite, faithful and completely reducible $G$-module. Assume $G \neq 1$, then there exits $v \in V$ such that $\lambda \cdot |G| \leq |v^G|^{\alpha}$.
\end{theorem}
\begin{proof}

We proceed by induction on $|G||V|$.

First we show that we may assume that $V$ is irreducible. Assume not, then $V=V_1\oplus V_2$ where $V_1,V_2$ are nontrivial $G$-submodules of $V$.

We know that $G$ acts completely reducibly on $V_1$ and $V_2$. Since $\bC_G(V_1) \nor G$, $\bC_G(V_1)$ acts completely reducible on $V_2$.

Let $m_1$ be the largest orbit size of $G$ on $V_1$, and let $m_2$ be the largest orbit size in the action of $\bC_G(V_1)$ on $V_2$.
Moreover, let $v_i\in V_i$ ($i=1,2$) be representatives of these orbits. Put $v=v_1+v_2$. Then $\bC_G(v)=\bC_G(v_1)\cap \bC_G(v_2)$ and hence
\begin{eqnarray*}
M&\geq&|v^G|=|G:(\bC_G(v_1)\cap \bC_G(v_2))|\\
 &=&|G:\bC_G(v_1)|\cdot|\bC_G(v_1):(\bC_G(v_1)\cap \bC_G(v_2))|\ =\ m_1\cdot|\bC_G(v_1):\bC_{\bC_G(v_1)}(v_2)|\\
 &=&m_1\cdot|v_2^{\bC_G(v_1)}|\  \geq\  m_1|v_2^{\bC_G(V_1)}|\ =\ m_1 m_2.
\end{eqnarray*}

Since $G >1$, either $G/\bC_G(V_1)>1$ or $\bC_G(V_1)>1$, by the inductive hypothesis we further conclude that

\[(*)\quad M^{\alpha}\geq m_1^{\alpha} m_2^{\alpha} \geq \lambda \cdot |G/\bC_G(V_1)|\cdot |\bC_G(V_1)| = \lambda \cdot |G|.\]

So, from now on let $V$ be irreducible.\\

We now assume that $V$ is not primitive. We hence assume that there exists a proper subgroup $L_1$ of $G$ and an irreducible $L_1$-submodule
$V_1$ of $V$ such that $V ={V_1}^G$. By transitivity of induction, we can choose $L_1$ to be a maximal subgroup of $G$.
In particular, $S \cong G/N$ is a primitive permutation group on a right transversal of $L_1$ in $G$, where $N$ is the normal
core of $L_1$ in $G$. Let $V_N=V_1 \oplus \cdots \oplus V_m$, where the $V_i$'s are irreducible $L_i$-modules where $L_i=\bN_G(V_i)$ and
$m>1$. We know $G/N$ primitively permutes the elements of $\{V_1, \dots, V_m\}$.

Define
\[N_i=\bC_N\left(\sum_{j=1}^{i-1}V_j\right) \Bigg/  \bC_N\left(\sum_{j=1}^{i}V_j\right)\]
for $i=1,\dots ,m$ and note that $N_1=N/\bC_N(V_1)$ and $N_m=\bC_N(\sum_{j=1}^{m-1}V_j)$. 
Then $|N| = \prod_{i=1}^{m} |N_i|$.
Clearly $N_i$ acts completely reducibly on $V_i$ for $i=1,\dots ,m$. Let $M_i$ be the largest orbit size of the action of $N_i$ on $V_i$
($i=1,\dots ,m$), and let $v_i\in V_i$ be representatives of the corresponding orbits for all $i$.
Thus, \[M_{N}\  \geq\ \left|\left(\sum_{i=1}^{m}v_i\right)^{N} \right| \ \geq\ \prod_{i=1}^{m} |v_i^{N_i}|\ =\ \prod_{i=1}^{m} M_i.\]


We define $\bar G=GN/N$. We denote the set of $m$ elements to be $\Omega$, we define $\Lambda$ a subset of $\Omega$ that contains all $i$ such that $N_i>1$. It is clear that $|N|=\prod_{i \in \Lambda}|N_i|$.

We now choose a vector $w \in V$ wisely so the orbit size of $w$ is big enough and it will satisfy the requirement.

Assume that $|\bar G| \leq \lambda^{|\Lambda|-1}$, we set $w=w_1+w_2+\cdots+w_m$ where $w_i=v_i$ if $N_i>1$ and $w_i=0$ if $N_i=1$.


Assume that $|\bar G| > \lambda^{|\Lambda|-1}$, then by Theorem ~\ref{thmperm2}, we may find  $\Delta \subseteq \Omega-\Lambda$ such that $|\bar G:\stab_{\bar G}(\Delta)|^{\alpha} \cdot \lambda^{|\Lambda|-1} \ge |\bar G|$. We set $w=w_1+w_2+\cdots+w_m$ where $w_i=v_i$ if $i \not\in \Delta$ and $w_i=0$ if $i \in \Delta$.\\

We consider the orbit size of $w$. We know that $|w^G|=|G:\bC_G(w)|$, and we define $M = |G:\bC_G(w)|$.

We further conclude that
\[\ M\ \geq\ |\bar G:\stab_{\bar G}(\Delta)| \cdot  \prod_{i \in \Lambda} M_i.\]

Thus, we have that


\begin{eqnarray*}
M^{\alpha} & \geq & |\bar G:\stab_{\bar G}(\Delta)|^{\alpha} \cdot \prod_{i \in \Lambda} M_i^{\alpha} \\
&\geq & |\bar G:\stab_{\bar G}(\Delta)|^{\alpha} \cdot \lambda^{|\Lambda|} \cdot \prod_{i \in \Lambda}|N_i|\ \\
&\geq & \lambda \cdot |\bar G| \cdot |N| =  \lambda |G|.
\end{eqnarray*}

Hence, now we may assume that $V$ is irreducible and primitive, and result follows by Proposition ~\ref{proporbit}.
\end{proof}

\begin{remark} We consider the following example. Assume $|V|=\FF_2^2$,  $G \cong \Gamma(2^2) \cong S_3$, where $|G|= 6$, the largest orbit of $G$ on $V$ is of size $3$, we see that the bound in Theorem ~\ref{thm1} is attained in this case. This shows that the orbit theorem we obtained here is almost the best one can get in general.
\end{remark}

\section{Gluck's conjecture and Navarro's conjecture} \label{sec:applications}
In this section, we will see various applications of the orbit theorem.

We first study a conjecture of Gluck that bounds the index of the Fitting subgroup of a solvable group. The following is the best general bound for this conjecture. 

\begin{theorem} \label{thm4}
Let $G$ is a finite solvable group. Then $|G: \bF(G)| \leq b^{\alpha}(G)$.
\end{theorem}
\begin{proof}
Let $U=\bF(G)/\Phi(G))$ and $\bar G=G/\bF(G)$. $U$ is a faithful and completely reducible $\bar G$-module by Gasch\"utz's theorem ~\cite[Theorem 1.12]{manz/wolf}. Let $V=\Irr(\bF(G)/\Phi(G))$ and $V$ a faithful and completely reducible $\bar G$-module by ~\cite[Proposition 12.1]{manz/wolf}. By Theorem ~\ref{thm1}, there exists $\gamma \in V$ such that $\bar{I}=I_{\bar{G}}(\gamma)= \{\bar{g} \in \bar{G} \ |\ \gamma^{\bar{g}}=\alpha \}$ satisfies $\lambda \cdot |\bar G| \leq |\bar G: \bar I|^{\alpha}$. Consider $\alpha$ as a character of $\bF(G)$ with kernel containing $\Phi(G)$. Let $I$ be the preimage of $\bar{I}$ in $G$. Now $I=I_G(\alpha)= \{g \in G \ | \ \gamma^g=\gamma \}$. Take $\mu \in \Irr(I|\gamma)$. Now $\psi=\mu^G \in \Irr(G)$. Thus, we have $|G:\bF(G)| \leq |G:I|^{\alpha} \leq \psi(1)^{\alpha} \leq b(G)^{\alpha}$.
\end{proof}

The following are some other applications of Theorem ~\ref{thm4}.

\begin{cor} \label{cor1}
Suppose that $G$ is a finite solvable group. Then $G$ has an abelian subgroup $A$ such that $|G: A| \leq b^{4+\alpha}(G)$.
\end{cor}


\begin{theorem} \label{Fittinggeneral}
  Suppose that $G$ is a finite solvable group and $V$ is a finite, faithful and completely reducible $G$-module. If $G \neq 1$, then there exists $v \in V$ such that $|G:\bF(G)| \leq |v^G|^\beta$.
\end{theorem}
\begin{proof}
Since $|G| \leq |v^G|^{\alpha}$, we know that $|G:\bF(G)| \leq |\bF(G)|^{2.25}$ and thus $|G| \leq |\bF(G)|^{3.25}$. So we have that
\[|G:\bF(G)| \leq |G|^{1-\frac 1 {3.25}} \leq |G|^{\frac {2.25} {3.25}} \leq M^{{\frac {2.25} {3.25}} \cdot {\alpha}}=M^{\beta}.\]
\end{proof}

We next show that if one goes up to the second Fitting group, we shall get something much stronger. In short, we can find a better bound for $|G:\bF_2(G)|$ with respect to $b(G)$.

\begin{theorem} \label{GluckconjectureFitting}
  Let $G$ be a finite solvable group. Then $|G:\bF_2(G)| \leq b(G)^{\beta}$.
\end{theorem}
\begin{proof}
Let $U=\bF(G)/\Phi(G)$ and $\bar G=G/\bF(G)$. $U$ is a faithful and completely reducible $\bar G$-module by Gasch\"utz's theorem ~\cite[Theorem 1.12]{manz/wolf}. Let $V=\Irr(\bF(G)/\Phi(G))$ and $V$ is a faithful and completely reducible $\bar G$-module by ~\cite[Proposition 12.1]{manz/wolf}. By Theorem ~\ref{Fittinggeneral}, there exists $\gamma \in V$ such that $\bar{I}=I_{\bar{G}}(\gamma)= \{\bar{g} \in \bar{G} \ | \ \gamma^{\bar{g}}=\gamma \}$ satisfies $|\bar{G}/\bar{I}|^{\beta} \geq {|\bar{G}/\bF(\bar{G})|}$. Consider $\gamma$ as a character of $\bF(G)$ with kernel containing $\Phi(G)$. Let $I$ be the preimage of $\bar{I}$ in $G$. Now $I=I_G(\gamma)= \{g \in G \ | \ \gamma^g=\gamma \}$. Take $\mu \in \Irr(I|\gamma)$. Now $\psi=\mu^G \in \Irr(G)$. Thus, we have $|G:\bF_2(G)| \leq |G:I|^{\beta} \leq \psi(1)^{\beta} \leq b(G)^{\beta}$.
\end{proof}




We next study a related question of Navarro. Navarro has conjectured that for solvable groups, $\prod_{p \in \pi(G)} b(G_p) \leq b(G)$ where $\pi(G)$ is the set of prime divisors of $|G|$ and $G_p \in \Syl_p(G)$ (see ~\cite[Conjecture 5]{Moret1}). This question is pretty challenging and we believe that if it is correct, then the proof will involve many delicate analyses. Not much has been done on this problem except for using some global estimations. Here we use a global approach and see what kind of upper bound can we get. The following is an easy consequence of Theorem ~\ref{thm4}, and it improves ~\cite[Corollary 2.8]{MOWOLF}. 

\begin{theorem}
If $G$ is solvable, then $\prod_{p \in \pi(G)} b(G_p) \leq b^{\alpha+1}(G)$ where $G_p \in \Syl_p(G)$.
\end{theorem}

\begin{remark}
For solvable linear groups, it is not always true that $|G| \leq M^2$, see for example ~\cite{WOLF1}. Thus, it seems to the author that one cannot prove Gluck's conjecture using orbit theorem alone.
\end{remark}

\section{Further generalization to arbitrary groups} \label{sec:further}

In this section, we further generalize the results to arbitrary linear groups.

Given a chief series
$$\Delta: 1=G_0< G_1< \cdots G_n=G$$ of a finite group $G$.
Let
${\rm Ord}_{\mathcal{S}}(G)$ denote the product of orders of all solvable chief factors $G_i/G_{i-1}$ in $\Delta$. Let $\mu(G)$ be the number of nonabelian chief factors in $\Delta$. Clearly, the constants ${\rm Ord}_{\mathcal{S}}(G)$ and $\mu(G)$ are independent of the choice of chief series $\Delta$ of $G$. We remark that as an application of Theorem ~\ref{thm4}, one can strengthen the solvable case of ~\cite[Theorem 4.7]{QianYangLarge}.

\begin{theorem}  \label{t101}
Let a finite group $G$ act faithfully on a finite group $V$,
and $M$ be the largest orbit size in the action of $G$ on
$V$. Then any one of the following conditions guarantees that
$$ M^{\alpha} \geq 2^{\mu(G)}\cdot {\rm Ord}_{\mathcal{S}}(G)$$,

{\rm (1)} $V$ is a $p$-group and $O_p(G)=1$ for some prime $p$;

{\rm (2)} $V$ is  a completely reducible $G$-module, possibly of mixed characteristic.
\end{theorem}
\begin{proof}

(1) Suppose first that $V$ is a $p$-group and $O_p(G)=1$.
By ~\cite[Proposition 3.5]{QianYangLarge},
$G$ has a solvable subgroup  $H$ with $O_p(H)=1$
such that ${\rm Ord}_{\mathcal{S}}(H)\geq 2^{\mu(G)}\cdot {\rm Ord}_{\mathcal{S}}(G)$.
Assume that $H< G$.
By induction, there exists an element $v\in V$ such that
$|H: C_H(v)|^{\alpha}\geq {\rm Ord}_{\mathcal{S}}(H)$.
Then $$M^{\alpha}\geq |G: C_G(v)|^{\alpha}\geq |H: C_H(v)|^{\alpha}
\geq {\rm Ord}_{\mathcal{S}}(H)
\geq 2^{\mu(G)}\cdot{\rm Ord}_{\mathcal{S}}(G),$$
and we are done.
Hence, we may assume that
$G=H$ is solvable.

Set $U=G\ltimes V$. It is clear that $F(U)$ is nilpotent, so we can write $F(K) = P \times Q$, where $P$ is a $p$-group and $Q$ is a $p'$-group. Note that both $P$ and $Q$ are normal in $U$, which implies that $PV/V$ is a normal $p$-subgroup of $U/V \cong G$ and $[Q,V] \leq Q \cap V = 1$.
Now the assumptions that $O_p(G) = 1$ and $\bC_G(V) = 1$ guarantee that $P \leq V$ and $Q = 1$, so $F(U) = P \leq  V$. But the converse containment is clear, so $\Phi(U) \leq F(U) = V$. If $\Phi(U) = V$, then $U = GV = G \Phi(U)$, which forces $U = G$ and hence $V = 1$, a trivial case. Thus we have $F(U)=V$ and $\Phi(U)< V$.

Set $$\overline{G}=G\Phi(U)/\Phi(U), \overline{V}=V/\Phi(U), \overline{U}=U/\Phi(U) =\overline{G} \ltimes \overline{V}.$$
Observe that $\overline{G}$ also acts faithfully on $\overline{V}$. We note this assertion is equivalent to saying that $G$ acts faithfully on $\overline{V}$. Let $C = \bC_G(\overline{V})$, and take $A$ to be a $p'$-subgroup of $C$. Then $[V,A] \leq \Phi(U)$. Since $A$ acts coprimely on $V$, we have $V = [V, A] \bC_V(A) = \Phi(U) \bC_V(A)$ and thus $U = GV = \Phi(U) G \bC_V(A)$. This forces $GV = G\bC_V(A)$ and hence $V = \bC_V(A)$. Note that we are assuming that $\bC_G(V) = 1$, so $A = 1$, which means that $C$ is a $p$-group. But $O_p(G) = 1$ by hypotheses, so $C = 1$, as required.


Assume that $\Phi(U)>1$.
By induction,  there exists an element $v\in V$ such
that $|\overline{G}: C_{\overline{G}}(\overline{v})|^{\alpha}
\geq {\rm Ord}_{\mathcal{S}}(\overline{G})$.
Since $$|G: C_G(v)|\geq |\overline{G}: C_{\overline{G}}(\overline{v})|\,\,{\rm  and}\,\,
G\cong \overline{G},$$
we get that $M^{\alpha} \geq |G: C_G(v)|^{\alpha}\geq {\rm Ord}_{\mathcal{S}}(G)$, and we are done.
Therefore, we may assume that $\Phi(U)=1$.
Now $V$ is a faithful and  completely  reducible $G$-module over a field of characteristic $p$. Thus the result follows by Theorem ~\ref{thm1}.

\bigskip

(2) Suppose that  $V$ is a finite, faithful and completely reducible $G$-module.
We work by induction on $|G|+|V|$.
Assume that $V=V_1\oplus V_2$, where $V_1, V_2$ are nontrivial $G$-submodules of $V$.
Let $K=C_G(V_1)$.
Observe that
$V_1$ is a faithfully and completely reducible $G/K$-module, while
 $V_2$ is a faithfully and completely reducible $K$-module (note that $K$ is normal in $G$).
By induction,
there exist $v_1\in V_1$ and $v_2\in V_2$ such that
$$|G: C_G(v_1)|^{\alpha}=|G/K : C_{G/K}(v_1)|^{\alpha}
\geq 2^{\mu(G/K)}\cdot{\rm Ord}_{\mathcal{S}}(G/K),$$
$$ |K: C_K(v_2)|^{\alpha}\geq 2^{\mu(K)}\cdot{\rm Ord}_{\mathcal{S}}(K).$$
Since ${\rm Ord}_{\mathcal{S}}(G/K) \cdot {\rm Ord}_{\mathcal{S}}(K)
\geq {\rm Ord}_{\mathcal{S}}(G)$ and $\mu(G/K)+\mu(K)\geq \mu(G)$,
we get that
$$\begin{array}{ccl}M^{\alpha} &\geq & |G: C_G(v_1+v_2)|^{\alpha}= |G: C_G(v_1)\cap C_G(v_2)|^{\alpha}\\
&=&|G: C_G(v_1)|^{\alpha}|C_G(v_1): C_{C_G(v_1)}(v_2)|^{\alpha}\\
&= & |G:C_G(v_1)|^{\alpha}|K: C_K(v_2)|^{\alpha}\\
&=&2^{\mu(G/K)}\cdot{\rm Ord}_{\mathcal{S}}(G/K)\cdot 2^{\mu(K)}\cdot{\rm Ord}_{\mathcal{S}}(K)\\
&\geq & 2^{\mu(G)}\cdot{\rm Ord}_{\mathcal{S}}(G),\end{array}$$
and we are done.
Hence, we may assume that
$V$ is an irreducible $G$-module over a finite field with $p$ elements.
Clearly $O_p(G)=1$ because $G$ acts faithfully on $V$.
Now the required result follows by (1).
\end{proof}

\begin{lemma} \label{l406}
Let $E=W_1\times \cdots \times W_n$
be a direct product of isomorphic nonabelian
simple groups $W_i$. Assume that  $G=\Aut(E)$ and $E$ is minimal normal in $G$.
Then  ${b}(E)^{\alpha}\geq 2^{\mu(G)}\cdot {\rm Ord}_{\mathcal{S}}(G)$.
\end{lemma}
\begin{proof}
Set $W\cong W_1$. We claim first that  ${b}(W)\geq 2|\Out(W)|$.
Assume that  $W$ is isomorphic to a sporadic simple group or  the Tits group
or an alternating group ${\rm Alt}_n$ with $n\not=6$.
Then ${\rm Ord}_{\mathcal{S}}(G)\leq 2$ and thus ${b}(W)\geq 5\geq 2|\Out(W)|$.
For $W\cong {\rm Alt}_6$, the claim  follows by a directly calculation.
Assume that  $W$ is of Lie type over a finite field of characteristic $p$.
Let $\St$ be the Steinberg character of $W$. Then
$\St(1)=|W|_p$, and ~\cite[Lemma 3.1]{QianYangLarge} implies the claim.
Now the claim implies that ${b}(E)\geq  2^n|{\rm Out}(W)|^n$.

Let $D=\bigcap_{i=1}^n N_G(W_i)$. Note that $\bigcap_{1\leq i\leq n} C_D(W_i)W_i=E$.
Since all $W_i$ are normal in $D$, we get that
$$\begin{array}{ccl}
D/E &=& D/\bigcap_{1\leq i\leq n} W_iC_D(W_i)\leq D/W_1C_D(W_1) \times \cdots
\times D/W_nC_D(W_n)\\
&=&N_D(W_1)/W_1C_D(W_1)\times \cdots \times N_D(W_n)/W_nC_D(W_i)\\
&\leq & {\rm Out}(W_1)\times \cdots \times {\rm Out}(W_n).\end{array}$$
In particular, $D/E$ is solvable with   $|D/E|\leq |\Out(W)|^n$.
Since $G$ acts on $\{W_1, \ldots, W_n\}$ with the kernel $D$,
$G/D$ is a permutation group of degree $n$.
Observe that ${\rm Ord}_{\mathcal{S}}(G)={\rm Ord}_{\mathcal{S}}(G/E)$ and
$\mu(G)=1+\mu(G/D)$.
By ~\cite[Theorem 4.1]{QianYangLarge}, we have that
$2^{\mu(G/D)}\cdot {\rm Ord}_{\mathcal{S}}(G/D)\leq \lambda^{n-1}$,
and this implies that
$$(2^{\mu(G)}\cdot {\rm Ord}_{\mathcal{S}}(G/D))^{1/{\alpha}}
\leq (2\lambda^{n-1})^{1/{\alpha}}< 2^n.$$
Thus $${b}(E)^{\alpha} \geq (2^n|\Out(W)|^n)^{\alpha} \geq  (2^{\mu(G)}\cdot{\rm Ord}_{\mathcal{S}}(G/D))|D/E|\geq 2^{\mu(G)}\cdot{\rm Ord}_{\mathcal{S}}(G),$$
and we are done.
\end{proof}

The following result generalizes Theorem ~\ref{thm4}.

\begin{theorem} \label{thm4application}
Let $G$ be a finite group.  Then
$2^{\mu(G)}\cdot {\rm Ord}_{\mathcal{S}}(G/\bF(G))) \leq b^{\alpha}(G)$.
\end{theorem}
\begin{proof}
Clearly  for any subgroup or quotient group $L$ of $G$,
we have ${b}(L) \leq {b}(G)$.
Since $(G/\Phi(G))/(\bF(G/\Phi(G)))\cong G/\bF(G)$,
by induction we may assume that
$\Phi(G)=1$.

\bigskip

Case 1. Suppose that $G$ contains a nonsolvable minimal normal subgroup $D$.

Write $C_G(D)=C$.
By Lemma \ref{l406}, there exists $\mu\in {\rm Irr}(D)$  such that
$$\mu(1)^{\alpha} \geq  2^{\mu(G/C)}\cdot{\rm Ord}_{\mathcal{S}}(G/C).$$
Clearly $$\bF(C)=\bF(G)\leq C<G.$$
By induction, there exists $\nu\in {\rm Irr}(C)$  such that
$$\nu(1)^{\alpha}\geq 2^{\mu(C/\bF(G))}\cdot {\rm Ord}_{\mathcal{S}}(C/\bF(G)).$$
Since $DC=D\times C$, we see that  $\mu\nu\in {\rm Irr}(DC)$ has degree $\mu(1)\nu(1)$.

Note that ${\rm Ord}_{\mathcal{S}}(G/C)\cdot {\rm Ord}_{\mathcal{S}}(C/\bF(G))\geq {\rm Ord}_{\mathcal{S}}(G/\bF(G))$ by ~\cite[Lemma 2.3(2)]{QianYangLarge} and that
$\mu(G/C)+\mu(C/\bF(G))\geq \mu(G/\bF(G))$. We get  that
$$\begin{array}{ccl}
{b}^{\alpha}(G) &\geq &  {b}^{\alpha}(DC) \geq (\mu(1)\nu(1))^{\alpha}\\
&\geq &2^{\mu(G/C)}\cdot{\rm Ord}_{\mathcal{S}}(G/C)\cdot 2^{\mu(C/\bF(G))}\cdot {\rm Ord}_{\mathcal{S}}(C/\bF(G)\\
&\geq&2^{\mu(G)}\cdot {\rm Ord}_{\mathcal{S}}(G/\bF(G)),\end{array}$$
and we are done.

\bigskip

Case 2. Suppose that all minimal normal subgroups of $G$ are solvable. Then the generalized Fitting subgroup of $G$ equals to  $\bF(G)$.
By $\Phi(G)=1$, $G/\bF(G)$ acts faithfully and completely reducibly on $\bF(G)$ and on ${\rm Irr}(\bF(G))$,
now the required result follows by Theorem \ref{t101}.
\end{proof}




\bigskip

\noindent
\textbf{Acknowledgement} \label{sec:Acknowledgement}

This work was partially supported a grant from the Simons Foundation (No 499532).\\







\begin{thebibliography}{19}
\bibitem {DOLFI} {S. Dolfi}, `Large orbits in coprime actions of finite solvable groups', {Trans. Amer. Math. Soc.} 360, no. 1 (2008), 135-152.

\bibitem {Gluckgeneral} {J. Cossey, Z. Halasi, A. Mar\'oti, and H.\,N. Nguyen}, `On a conjecture of Gluck', Math. Z. 279 (2015), no. 3-4, 1067-1080.

\bibitem{Dixon-Mortimer} J.\,D. Dixon and B. Mortimer, Permutation Groups, Graduate Texts in Mathematics 163, Springer-Verlag, New York, 1996.

\bibitem{DolfiJ}{S. Dolfi and E. Jabara}, `Large character degrees of solvable groups with abelian Sylow $2$-subgroups', {J. Algebra} 313 (2007), 687-694.
\bibitem{AE1} {A. Espuelas}, `Large character degree of groups of odd order', {Illinois J. Math} 35 (1991), 499-505.
\bibitem{GLS2} {D. Gorenstein, R. Lyons, and R. Solomon}, The Classification of the Finite Simple Groups, Number 2, {AMS}, Providince, RI, 1996.

\bibitem{GAP} The GAP Group, GAP -- Groups, Algorithms, and Programming, Version
4.11.0; 2020. (\url{https://www.gap-system.org})

\bibitem{GLUCK} {D. Gluck}, `The largest irreducible character degree of a finite group', {Canad. J. Math.} 37 (3) (1985), 442-451.


\bibitem{IMIB} {I.\,M. Isaacs}, Character Theory of Finite Groups, Dover, New York, 1994.

\bibitem{QianYangLarge} {G. Qian and Y. Yang}, `Large orbit sizes in finite group actions', J. Pure Appl. Algebra 225 (1) (2021), 106458, 17 pp.

\bibitem{manz/wolf} {O. Manz and T.\,R. Wolf}, Representations of solvable groups, {Cambridge University Press}, 1993.

\bibitem{Moret1} {A. Moret\'o}, `Characters of $p$-groups and Sylow $p$-subgroups', Groups St. Andrews 2001 in Oxford, Cambridge University Press, Cambridge, 412-421.



\bibitem{MOWOLF} {A. Moret\'o and T.\,R. Wolf}, `Orbit sizes, character degrees and Sylow subgroups', {Adv. Math.} 184 (2004), 18-36.


\bibitem{Seress} {A. Seress}, `The minimal base size of primitive solvable permutation groups', J. Lond. Math. Soc. 53 (1996), 243-255.

\bibitem{WOLF1} {T.\,R. Wolf}, `Large orbits of supersolvable linear groups', {J. Algebra} 215 (1) (1999), 235-247.



\bibitem{YY2} {Y. Yang}, `Regular orbits of finite primitive solvable groups', {J. Algebra} 323 (2010), 2735-2755.
\bibitem{YY3} {Y. Yang}, `Regular orbits of finite primitive solvable groups, II', {J. Algebra} 341 (2011), 23-34.
\bibitem{YY4} {Y. Yang}, `Large character degrees of solvable $3'$-groups', {Proc. Amer. Math. Soc.} 139 (2011), 3171-3173.
\bibitem{YY8} {Y. Yang}, `Large orbits of odd-order subgroups of solvable linear groups', {J. Algebra} 351 (2012), 220-234.
\bibitem{YYodd} {Y. Yang}, `Arithmetical conditions of orbit sizes of odd order linear groups', Israel J. Math. 237 (2020), 1-14.

\end{thebibliography}
\end{document}